\documentclass[10pt,oneside,a4paper,draft]{article} 	

\usepackage[fleqn]{amsmath} 						
\usepackage{amsthm,amsfonts,latexsym,amssymb,amscd} 	
\usepackage[mathscr]{eucal}   						
\usepackage[all,2cell]{xy} \UseAllTwocells 		
\usepackage{txfonts}

\usepackage[draft=false]{hyperref} 								

\newcommand{\KK}{{\mathbb{K}}}

\DeclareFontFamily{U}{rsfs}{\skewchar\font127 }
\DeclareFontShape{U}{rsfs}{m}{n}{%
   <5> <6> rsfs5
   <7> rsfs7
   <8> <9> <10> <10.95> <12> <14.4> <17.28> <20.74> <24.88> rsfs10
}{}
\DeclareSymbolFont{rsfs}{U}{rsfs}{m}{n} 
\DeclareSymbolFontAlphabet{\scr}{rsfs}

\DeclareMathOperator{\Aut}{Aut}

\DeclareMathOperator{\Span}{span}

\renewcommand{\emph}{\textbf} 										

\newcommand{\hlink}[2]{\href{#1}{\texttt{#2}}} 

\newcommand{\xqedhere}[2]{%
  \rlap{\hbox to#1{\hfil\llap{\ensuremath{#2}}}}}

\newcommand{\xqed}[1]{%
  \leavevmode\unskip\penalty9999 \hbox{}\nobreak\hfill
  \quad\hbox{\ensuremath{#1}}}

\theoremstyle{plain}
\newtheorem{theorem}{Theorem}[section]							
\newtheorem{corollary}[theorem]{Corollary}
\newtheorem{lemma}[theorem]{Lemma}

\newtheorem{definition}[theorem]{Definition}

\theoremstyle{definition} 
\newtheorem{remark}[theorem]{Remark}
\newtheorem{example}[theorem]{Example}  

\numberwithin{equation}{section}  		
\title{\textbf{Spectral Theory on Commutative 
Kre\u\i n C*-algebras}
}

\author{\normalsize 
Pichkitti Bannangkoon$^a$\footnote{Current Address: Department of Mathematics, Penn State University, University Park, State College, PA 16802,  USA.} , 
Paolo Bertozzini$^b$, 
Wicharn Lewkeeratiyutkul$^a$  
\\ 
\normalsize $^a$\textit{Department of Mathematics and Computer Science, Faculty of Science,}
\\
\normalsize \textit{Chulalongkorn University, Bangkok 10330, Thailand}
\\ 
\normalsize e-mails: \texttt{pb\_math@yahoo.com}  \quad \texttt{Wicharn.L@chula.ac.th}
\\
\normalsize $^b$\textit{Department of Mathematics and Statistics, Faculty of Science and Technology,}
\\
\normalsize \textit{Thammasat University, Bangkok 12121, Thailand} 
\\ 
\normalsize e-mail: \texttt{paolo.th@gmail.com}
}

\date{\normalsize{08 December 2013}\footnote{This is a reformatted version for arXiv of an original paper presented at the 
``$14^\text{th}$ Annual Thai Mathematical Meeting'' in Prince of Songkla University Pattani Campus on 5-6 March 2009 and appeared in: 
\textit{$14^\text{th}$ Annual Meeting in Mathematics 2009, Collection of Full Papers} 38-55.}.
}

\begin{document}

\maketitle

\begin{abstract} \noindent 
A Banach involutive algebra $(A,*)$ is called a Kre\u\i n C*-algebra if there is a fundamental symmetry of $(A,*)$, i.e., $\alpha\in \Aut(A,*)$  such that $\alpha^2=1_A$  and $\|\alpha(x^*)x\|=\|x\|^2$  for all $x\in A$.
	Using $\alpha$, we can decompose $A=A_+\oplus A_-$  where $A_+=\{x\in A\mid\alpha(x)=x\}$  and $A_-=\{x\in A\mid\alpha(x)=-x\}$. The even part $A_+$  is a C*-algebra and the odd part $A_-$ is a Hilbert C*-bimodule over the even part $A_+$. The ultimate goal is to develop a spectral theory for commutative unital Kre\u\i n C*-algebras when the odd part is a symmetric imprimitivity bimodule over the even part and there exists a suitable ``exchange symmetry'' $\epsilon$ between $A_+$ and $A_-$. 

\smallskip


\smallskip

\noindent
\emph{MSC-2010:} 						
					47B50, 
					46C20. 
\end{abstract}




\section{Introduction}

The study of vector spaces equipped with ``inner products'' that are non-necessarily positive-definite has always been a theme of extreme importance in relativistic physics starting probably with the work of H.Minkowski~\cite{Ra}. Complex vector spaces with indefinite sesquiliner forms have been introduced in relativistic quantum field theory by P.Dirac~\cite{Di}, W.Pauli~\cite{Pau} and then used by S.Gupta~\cite{Gu} K.Bleuler~\cite{Bl}, although their mathematical definition has been given later by L.Pontrjagin~\cite{Po}. Complete ``indefinite'' inner product spaces, called Kre\u\i n spaces, have been introduced by Ju.Ginzburg~\cite{Gi} and E.Scheibe~\cite{Sc} and the study of their properties has been undertaken by several Russian mathematicians. 

Although algebras of continuous operators on Kre\u\i n spaces have been around for some time, a first definition of abstract Kre\u\i n C*-algebra has been provided only recently by K.Kawamura~\cite{Ka, Ka2}. Kre\u\i n C*-algebras are somehow expected to play some role in a ``semi-Riemannian'' version of A.Connes non-commutative geometry~\cite{Co} (see A.Strohmaier~\cite{St}, M.Pasche-A.Rennie-R.Verch~\cite{Pas} for details) and for this reason it should be of some interest to develope a spectral theory that is suitable for them. 

It is the purpose of this paper to introduce a simple spectral theory for the special class of Kre\u\i n C*-algebras that decompose, via a fundamental symmetry, in direct sum $A=A_+\oplus A_-$ with $A_+$ commutative C*-algebra and $A_-$ symmetric imprimitivity (anti-)Hilbert C*-bimodule over 
$A_+$ and that are equipped with an odd symmetry exchanging $A_+$ and $A_-$.  
Our main result (Theorem~\ref{thm2}) is that every such Kre\u\i n C*-algebra turns out to be isomorphic (via Gel'fand tranform) to an algebra of continous functions with values in a very elementary Kre\u\i n C*-algebra $\KK$ defined in detail in Theorem~\ref{tm2}. 

The main result presented here can actually be also obtained in at least a few other ways that we briefly describe here below:
\begin{itemize}
\item
For a symmetric imprimitivity commutative unital Kre\u\i n C*-algebra $A$, the even part $A_+$, as a commutative unital C*-algebra, is isomorphic to the algebra of sections of a trivial complex line bundle over the Gel'fand spectrum $\Omega(A_+)$. 
The odd part $A_-$ is a symmetric imprimitivity Hilbert C*-bimodule over $A_+$ and, making use of the spectral theorem for imprimitivity Hilbert 
C*-bimodules developed in~\cite{BCL}, it is isomorphic to the bimodule of sections of a complex line bundle over the same $\Omega(A_+)$. Under the existence of an odd symmetry on $A$, the Witney sum of the previous two line bundles turns out to be a bundle of rank-one Krein C*-algebras isomorphic to $\mathbb K$ over $\Omega(A_+)$. 
\item
Although we are not aware now of a specific reference to a Gel'fand theorem, once a specific fundamental symmetry/odd symmetry 
$\alpha,\epsilon$ has been choosen on $A$, the Kre\u\i n C*-algebra becomes completely equivalent to a $\mathbb Z_2$-graded C*-algebra with commutative even part and for such elementary C*-algebras spectral results are for sure obtainable as special cases from the general theory of 
C*-dynamical systems. 
\item
To every unital Kre\u\i n C*-algebra equipped with a given fundamental symmetry, we can always associate a C*-category with two objects. In the case of imprimitivity commutative unital Kre\u\i n C*-algebras, such C*-category will be commutative and full according to the definition provided 
in~\cite{Be} and our spectral theorem can be recovered as a trivial application of the general spectral theory for commutative full C*-categories developed in~\cite{Be}. 
\end{itemize}
The techniques utilized here in the proof of this result are essentially an adaptation of those developed in~\cite{Be} for a spectral theory of commutative full C*-categories. 
Our choice to develope a completely independent proof of the result can be justified from the desire to test and caliber some of the general techniques introduced in~\cite{Be} in a simple situation that in the near future might be used as a ``laboratory'' for 
non-commutative extensions of the spectral theorem. It is expected that more powerful spectral theories for wider classes of 
Kre\u\i n C*-algebras might be developed using a Kre\u\i n version of the spaceoid Fell bundle introduced in~\cite{Be}, we hope to address this more general problem in other preprints. 

\section{Preliminaries}

For the convenience and usefulness of the reader, we provide here some background material and beneficial theorems on the theory of C*-algebras and Hilbert C*-modules (see~\cite{Mu}~\cite{BCL} for all the details). 

\begin{definition}
An \emph{algebra} over the complex numbers is a complex vector space $A$ equipped with a binary operation (called product) 
$\cdot:A\times A\rightarrow A,\quad \cdot: (x,y)\mapsto x\cdot y$, 
that is bilinear.

The algebra is called \emph{associative} if:
$x\cdot(y\cdot z)=(x\cdot y)\cdot z$.

The algebra is called \emph{unital} if:
$\exists 1_A\in A$, such that $\forall x\in A, \ x\cdot 1_A=1_A\cdot x=x$. 
\end{definition}

\begin{definition}
An algebra is called \emph{involutive} or also \emph{$*$-algebra} if it is equipped with a function $*:A\rightarrow A$, such that:
\begin{align*}
  &(x^*)^*=x, \quad \forall x\in A,\\
  &(x\cdot y)^*=y^*\cdot x^*, \quad \forall x,y\in A,\\
  &(\alpha x+\beta y)^*=\bar{\alpha}x^*+\bar{\beta}y^*, \quad \forall\alpha, \beta\in \mathbb{C}, \ \forall x,y \in A.
\end{align*}
\end{definition}

\begin{definition}
A \emph{normed algebra} is an algebra $A$ that is also a normed space and that satisfies the property:
$\|x\cdot y\|\leq \|x\|\cdot\|y\|$, $\forall x,y \in A$. 
A \emph{Banach algebra} is a normed algebra that, as a normed space, is complete.
\end{definition}

\begin{definition}
A \emph{pre-C*-algebra} is and involutive algebra $A$ that is also a normed algebra that satisfies the property 
$\|x^*x\|=\|x\|^2$, $\forall x\in A$. 
A \emph{C*-algebra} is a pre-C*-algebra that is also Banach algebra. 
\end{definition}

\begin{example} 
The followings are examples of C*-algebras:
\begin{enumerate}
	\item Let $X$ be a compact Hausdorff space. Then the space $C(X)$ of all continuous complex-valued functions on $X$ is a unital commutative 
C*-algebra with the following operations and norm: $\forall x\in X$, 
\begin{align*}
&(f+g)(x)=f(x)+g(x), 
& 
&(cf)(x)=cf(x), 
\\
&(fg)(x)=f(x)g(x), 
&
&(f^*)(x)=\overline{f(x)}, 
\\
&\|f\|=\sup\{|f(x)| : x\in X\}.
\end{align*}	 
	\item The space $B(H)$ of bounded linear maps on a Hilbert space $H$ is a unital C*-algebra that is non-commutative if 
	$\dim(H)>1$, with the following operations and norm: 
\begin{align*}	
&(T+S)(x)=T(x)+S(x), 
&
&(cT)(x)=cT(x), 
\\ 
&(T\circ S)(x)=T(S(x)),
& 
&\|T\|=\sup_{x\neq 0}\,\frac{\|T(x)\|}{\|x\|}, 
\end{align*}
the involution $T^*$ of $T$ is the adjoint of $T$:	\  						 
$\langle Tx\mid y \rangle=\langle x\mid T^*y\rangle, \ \forall x,y\in H$. 
\xqedhere{2.4cm}{\lrcorner}
\end{enumerate}
\end{example}

\begin{definition}
A \emph{homomorphism} from an algebra $A$ to an algebra $B$ is a linear map $\varphi:A\rightarrow B$ such that $$\varphi(ab)=\varphi(a)\varphi(b)\,\quad \forall a, b\in A.$$
A \emph{$*$-homomorphism} $\varphi:A\rightarrow B$ of $*$-algebras $A$ and $B$ is a homomorphism of algebras such that 
$\varphi(a^*)=\varphi(a)^* \quad \forall a\in A.$
If in addition $\varphi$ is a bijection, it is a \emph{$*$-isomorphism}. 
\end{definition}

\begin{definition}
A \textbf{character} on a commutative algebra $A$ is a nonzero homomorphism $\tau:A\rightarrow \mathbb{C}$.  
We denote by $\Omega(A)$ the set of characters on $A$.
\end{definition}

For each $a\in A$, define $\hat{a}:\Omega(A)\rightarrow \mathbb{C}$ by $\hat{a}(\tau)=\tau(a)$ for all $\tau\in \Omega(A)$.
Equip $\Omega(A)$ with the smallest topology on $\Omega(A)$ which makes each $\hat{a}$ continuous.

\begin{theorem}
If $A$ is a unital commutative Banach algebra, then $\Omega(A)$ is a compact Hausdorff space with respect to the topology defined above.
\end{theorem}

\begin{theorem}[Gel'fand-Mazur]
If $A$ is a unital Banach algebra in which every non-zero element is invertible, then $A=\mathbb{C}1$.
\end{theorem}

\begin{theorem}[Spectral theorem]
If $A$ is a non-zero unital commutative C*-algebra, then the Gelfand transform
$$\varphi:A\rightarrow C(\Omega(A)),\ a\mapsto \hat{a},\quad (\hat{a}:\Omega(A)\rightarrow \mathbb{C}, \ \tau\mapsto\tau(a))$$
is an isometric $*$-isomorphism.
\end{theorem}

\begin{definition}
A \emph{right $R$-module $E$} over a ring $R$ is an Abelian group $(E,+)$ equipped with an operation $\cdot:E\times R\rightarrow E$, $\cdot:(x,a)\mapsto x\cdot a$, of right multiplication by elements of the ring $R$, that satisfies the following properties:
\begin{align*}
  &x\cdot(a+b)=(x\cdot a)+(x\cdot b), \quad \forall x\in E, \forall a, b\in R,\\
  &(x+y)\cdot a=(x\cdot a)+(y\cdot a), \quad \forall x,y \in E, \forall a\in R,\\
  &x\cdot (ab)=(x\cdot a)\cdot b, \quad \forall x\in E, \forall a,b \in R.
\end{align*}
If the ring $R$ is unital, we will say that $E$ is a \emph{unital right $R$-module} if the additional property here below is satisfied:
$$x\cdot 1_R=x,\quad \forall x\in E.$$
Analogously, we can define a \emph{(unital) left $R$-module $E$} over a ring $R$. 
\end{definition}
Note that a unital module over a unital complex algebra is naturally a complex vector space. 

\begin{definition}
A \emph{right pre-Hilbert C*-module }$M_B$ over a unital C*-algebra $B$ is a unital right module over the unital ring $B$ that is equipped with a 
$B$-valued inner product $(x,y)\mapsto \langle x|y \rangle$ such that
\begin{align*}
  &\langle z\mid x+y \rangle_B=\langle z\mid x\rangle_B+\langle z\mid y \rangle_B, \quad\forall x, y, z\in M,\\
  &\langle z\mid x\cdot b\rangle_B=\langle z\mid x\rangle_Bb, \quad\forall b\in B,\\
  &\langle y\mid x\rangle_B=\langle x\mid y\rangle_B^*, \quad\forall x, y\in M,\\
  &\langle x\mid x\rangle_B\in B_+, \quad\forall x\in M,\\
  &\langle x\mid x\rangle_B=0_B\Rightarrow x=0_M.
\end{align*}
Analogously, a \emph{left pre-Hilbert C*-module} $_AM$ over a unital C*-algebra A is a unital left module $M$ over the unital ring $A$, that is equipped with an $A$-valued inner product $M\times M \rightarrow A$ denoted by $(x,y)\mapsto _A\langle x\mid y\rangle$. Here the $A$-linearity in on the first variable.
\end{definition}

\begin{remark}
A right (respectively left) pre-Hilbert C*-module $M_B$ over the C*-algebra $B$ is naturally equipped with a norm
$$
\|x\|_M:=\sqrt{\|\langle x\mid x\rangle_B\|_B}, \quad\forall x\in M.
\xqedhere{4.55cm}{\lrcorner}
$$
\end{remark}

\begin{definition}
A right (resp. left) \emph{Hilbert C*-module} is a right (resp. left) pre-Hilbert C*-module over a C*-algebra $B$ that is a Banach space with respect to the previous norm $\|\cdot\|_M$ (resp. $_M\|\cdot\|$).
\end{definition}

\begin{definition}
A right Hilbert C*-module $M_B$ is said to be \emph{full} if
$$
\langle M_B\mid M_B\rangle_B:= \overline{\Span\{\langle x\mid y\rangle_B \mid x,y\in M_B\}}=B,
$$
where the closure is in the norm topology of the C*-algebra $B$. A similar definition holds for a left Hilbert C*-module.
\end{definition}

We recall here the following well-known result, see for example in~\cite[Page 65]{FGV}:
\begin{lemma}
Let $M_B$ be a right Hilbert C*-module over a unital C*-algebra $B$. Then $M_B$ is full if and only if 
$\Span\{\langle x\mid y \rangle_B\mid x,y\in M_B\}=B$.
\end{lemma}

We also recall a few definitions and results on bimodules (see for example~\cite{BCL}). 
\begin{theorem} (\cite[Proposition 2.6]{BCL})
Let $M_A$ be a right unital Hilbert C*-module over a unital \hbox{C*-algebra} $A$ and $J\subset A$ an involutive ideal in $A$.
 
The set $MJ:=\{\sum_{j=1}^N x_ja_j\mid x_j\in M, a_j\in J, N\in \mathbb{N}_0\}$ is a submodule of $M$. 
The quotient module $M/(MJ)$ has a natural structure as a right Hilbert C*-module over the quotient C*-algebra $A/J$. If $M$ is full over $A$, 
also $M/(MJ)$ is full over $A/J$. 

A similar statement holds for a left Hilbert C*-module. 
\end{theorem}

Recall that a unital bimodule $_AM_B$ over two unital rings $A$ and $B$ is a left unital $A$-module and a right unital $B$-module such that 
$(a\cdot x)\cdot b=a\cdot (x\cdot b)$, for all $a\in A, b\in B$ and $x\in M$.

\begin{definition}
A \emph{pre-Hilbert C*-bimodule} $_AM_B$ over a pair of unital C*-algebras $A,B$ is a left pre-Hilbert C*-module over $A$ and a right 
pre-Hilbert C*-module over $B$ such that:
\begin{align*}
    &(a\cdot x)\cdot b=a\cdot(x\cdot b), \quad \forall a\in A, x\in M, b\in B,\\
    &\langle x\mid ay\rangle_B=\langle a^*x|y\rangle_B,  \quad\forall x,y \in M, \forall a \in A,\\
    &_A\langle xb\mid y\rangle= _A\langle x\mid yb^*\rangle,  \quad\forall x,y \in M, \forall b\in B.
\end{align*}
A \emph{Hilbert C*-bimodule} $_AM_B$ is a pre-Hilbert C*-bimodule over $A$ and $B$ that is simultaneously a left Hilbert C*-module over $A$ and a right Hilbert C*-module over $B$.
A full Hilbert C*-bimodule over the C*-algebras $A_B$ is said to be an \emph{imprimitivity bimodule} or an \emph{equivalence bimodule} if:
$$
{}_A\langle x\mid y\rangle\cdot z=x\cdot\langle y\mid z\rangle_B \quad\forall x, y, z\in M.
$$
\end{definition}

\begin{remark} (see for example~\cite[Remark 2.14]{BCL})
In an $A$-$B$ pre-Hilbert C*-bimodule there are two, usually different, norms, one as a left-C*-module over $A$ and one as a right-C*-bimodule over 
$B$:
$$
{}_M\|x\|:=\sqrt{\|_A\langle x\mid x\rangle\|_A},\quad \|x\|_M:=\sqrt{\|\langle x\mid x\rangle_B\|_B},\quad \forall x\in M,
$$
but this two norms coincide for an imprimitivity bimodule. 
In fact,
\begin{align*}
_M\|x\|^4&=\|_A\langle x\mid x\rangle\|^2_A=\|_A\langle x\mid x\rangle _A\langle x\mid x\rangle\|_A
=\|_A\langle x\langle x\mid x\rangle_B\mid x\rangle\|_A\\
&\leq \|\langle x\mid x\rangle_B\|_B\cdot _A\|\langle x\mid x\rangle\|_A=\|x\|^2_M\cdot _M\|x\|^2.
\xqedhere{6.15cm}{\lrcorner}
\end{align*}
\end{remark}

\begin{definition}
We say that a bimodule $_A M_A$ is a \emph{symmetric bimodule} if $ax=xa$, for all $x\in M$ and all $a\in A$. 
If $_AM_A$ is a Hilbert C*-bimodule, we say that it is a \emph{symmetric C*-bimodule} if it is symmetric as a bimodule and 
$_A\langle x\mid y\rangle=\langle y\mid x\rangle_A$, for all $x,y\in M$.
\end{definition}

\section{Main Results}

First of all, we review the definition of Kre\u\i n C*-algebra introduced by K.Kawamura~\cite{Ka2} and we explore some elementary properties of a 
Kre\u\i n C*-algebra.
\begin{definition} (K.Kawamura \cite[Definition 2]{Ka2})
A \textbf{Kre\u\i n C*-algebra} is a Banach $*$-algebra $A$ admitting at least one \textbf{fundamental symmetry} 
i.e. a $*$-automorphism $\phi:A\rightarrow A$ with $\phi\circ\phi=\iota_A$ such that 
$\|\phi(a^*)a\|=\|a\|^2$, for all $a\in A$.
\end{definition}
Note that, in general, a Kre\u\i n C*-algebra can admit several different fundamental symmetries. 

\begin{remark}
A C*-algebra $A$ is a Kre\u\i n C*-algebra with fundamental symmetry $\iota_A$ and so  a Kre\u\i n C*-algebra is a generalization of 
a C*-algebra. 
Let $(A,\alpha)$ be a Kre\u\i n C*-algebra with a given fundamental symmetry $\alpha$. Then we always have the decomposition 
$$
A=A_+\oplus A_-, \quad \text{where} \quad A_+=\{x\in A\mid\alpha(x)=x\},\ A_-=\{x\in A\mid\alpha(x)=-x\}.
$$
Indeed, if $a\in A_+\cap A_-$, then $a=\phi(a)=-a$ and so $a=0$. Moreover for all $a\in A$, 
$ a=\frac{a+\alpha(a)}2+\frac{a-\alpha(a)}2$ where $\frac{a+\alpha(a)}2\in A_+$ and $\frac{a-\alpha(a)}2\in A_-$.
\xqed{\lrcorner} 
\end{remark}

Every Kre\u\i n C*-algebra $A$ with a given fundamental symmetry $\alpha$ becomes naturally a C*-algebra, denoted here by 
$(A,\dag_\alpha)$ when equipped with a new involution $\dag_\alpha:A\to A$ as described in the following theorem. 

\begin{theorem}\label{tm6}
Let $(A,\alpha)$ be a Kre\u\i n C*-algebra with fundamental symmetry $\alpha$. Then $A$ becomes a C*-algebra when equipped with the new involution $\dag_\alpha$ defined by $\dag_\alpha: x\mapsto \alpha(x)^*$. Furthermore, the fundamental symmetry $\alpha$ is an automorphism of this 
C*-algebra and hence continuous in norm.
\end{theorem}
\begin{proof}
It is straightforward from the definition of Kre\u\i n C*-algebra to verify the first assertion. 
Since $\alpha(x^{\dag_\alpha})=\alpha(\alpha(x)^*)=\alpha(\alpha(x))^*=\alpha(x)^{\dag_\alpha}$, for all $x\in A$, we have that 
$\alpha$ is a $*$-homomorphism of the C*-algebra $(A,\dag_\alpha)$ into itself and hence it is continuous in norm. 
\end{proof}

\begin{remark}\label{re: c-k}
Note that given a $\dag$-automorphism $\alpha: A\to A$ of a C*-algebra $A$ (with involution denoted by $\dag$), we can naturally construct a Kre\u\i n C*-algebra $(A,*_\alpha)$ with involution $x^{*_\alpha}:=\alpha(x^\dag)$, for $x\in A$ and that $\alpha$ becomes a fundamental symmetry for this Kre\u\i n C*-algebra. 
\xqed{\lrcorner}
\end{remark}

Now we observe the algebraic structure of the even and odd part of $(A,\alpha)$.

\begin{theorem}\label{tm1}
Let $(A,\alpha)$ be a unital Kre\u\i n C*-algebra with a given fundamental symmetry $\alpha$.
Then $A_+$ is a unital C*-algebra and $A_-$ is a unital Hilbert C*-bimodule over $A_+$. 
\end{theorem}
\begin{proof}
Note first that on $A_+$ the two involutions $*$ and $\dag_\alpha$ coincide and since $A_+$ is closed under multiplication and involution, it is clearly a unital C*-algebra. 

We define right and left multiplications from $A_-\times A_+$ and $A_+\times A_-$ into $A_-$ as usual multiplications in $A$ and we define a pair of 
$A_+$-valued inner products from $A_-\times A_-$ into $A_+$ by 
\begin{equation}\label{eqip}
_{A_+}\langle x\mid y \rangle=xy^{\dag_\alpha}\quad  \text{and}  \quad 
\langle x\mid y \rangle_{A_+}=x^{\dag_\alpha}y
\end{equation} 
for all $x,y\in A_-$. With this definitions $A_-$ has the structure of unital Hilbert 
C*-bimodule on $A_+$. 
\end{proof}

Additionally the Hilbert C*-bimodule $A_-$ has the property,  
$$
_{A_+}\langle x\mid y \rangle z=(xy^{\dag_\alpha})z=x(y^{\dag_\alpha}z)=x\langle y\mid z \rangle_{A_+}. 
$$
for all $x,y,z\in A$.

For any Kre\u\i n C*-algebra equipped with a fundamental symmetry $\alpha$, the odd part $A_-$ in the fundamental decomposition  $A=A_+\oplus A_-$ is a Hilbert C*-bimodule but it is not in general an imprimitivity bimodule because $A_-$ might not be a full bimodule over $A_+$. In the following we will usually assume this further property. 

We will see later that the commutativity of $A_+$ and imprimitivity of $A_-$ play important roles in order to get our spectral theory for 
Kre\u\i n C*-algebras.
\begin{definition}
Let $(A,\alpha)$ be a Kre\u\i n C*-algebra with a given fundamental symmetry $\alpha$. A Kre\u\i n C*-algebra $(A,\alpha)$ is said to be 
\textbf{imprimitive} or \emph{full} if its odd part $A_-$ is an imprimitivity bimodule over $A_+$. 
We say that $(A,\alpha)$ is \textbf{rank-one} if $\dim A_+=\dim A_-=1$ as complex vector spaces. 
\end{definition}

\begin{theorem}
Let $(A,\alpha)$ be a Kre\u\i n C*-algebra with a given fundamental symmetry $\alpha$. The algebra $A$ is commutative if and only if the even part 
$A_+$ is a commutative unital C*-algebra and the odd part is a symmetric Hilbert C*-bimodule over $A_+$. 
In particular a rank-one Kre\u\i n C*-algebra is always commutative. 
\end{theorem}
\begin{proof}
If $A$ is commutative, clearly $A_-$ is a symmetric bimodule over the commutative algebra $A_+$ furthermore the inner products
defined in \eqref{eqip} above satisfy $_{A_+}\langle x\mid y\rangle=xy^{\dag_\alpha}=y^{\dag_\alpha}x=\langle y\mid x\rangle_{A_+}$ and hence $A_-$ is a symmetric C*-bimodule over the commutative C*-algebra $A_+$. 
Conversely, if $A_-$ is a symmetric C*-bimodule, 
$a_-b_-={}_{A_+}\langle a_-\mid b_-^{\dag_\alpha}\rangle=\langle b_-^{\dag_\alpha}\mid a_-\rangle_{A_+}=b_-a_-$, for all $a_-,b_-\in A_-$. 
Now, for all $a,b\in A$ with decompositions $a=a_++a_-$, $b=b_++b_-$ an easy computation shows: 
$ab=(a_++a_-)(b_++b_-)=a_+b_++a_+b_-+a_-b_++a_-+b_-=b_+a_++b_-a_++b_+a_-+b_-a_-=ba$.
\end{proof}

\begin{definition}
We will say that a Kre\u\i n C*-algebra is \emph{symmetric} if there exists an \emph{odd symmetry} i.e.~a linear map 
$\epsilon: A\to A$ such that $\epsilon \circ\epsilon=i_A$, $\epsilon(x^*)=-\epsilon(x)^*$, for all $x\in A$, 
$\epsilon(xy)=\epsilon(x)y=x\epsilon(y)$, for all $x,y\in A$ and if there exists a fundamental symmetry $\alpha$ such that 
$\epsilon\circ\alpha=-\alpha\circ\epsilon$ (in this case we say that the symmetry and the odd symmetry are compatible). 
\end{definition}

\begin{lemma}\label{lmanti}
If $\epsilon$ is an odd symmetry of a symmetric Kre\u\i n C*-algebra compatible with the symmetry $\alpha$, then $\epsilon$ is always isometric.
\end{lemma}
\begin{proof}
For all $x\in A$, $\|\epsilon(x)\|^2=\|\epsilon(x)^{\dag_\alpha}\epsilon(x)\|=\|\epsilon(x^{\dag_\alpha})\epsilon(x)\|
=\|\epsilon\circ\epsilon(x^{\dag_\alpha }x)\|=\|x^{\dag_\alpha}x\|=\|x\|^2$. 
\end{proof}

Next we give an example of rank-one unital full symmetric Kre\u\i n C*-algebra which has an important role in the proof of our spectral theory. 
Since, as we will see, in this case, the fundamental symmetry is necessarily unique, this is just a well-known example of commutative 
$\mathbb Z_2$-graded \hbox{C*-algebra}. Although all the properties described in the following three theorems are ``standard'' from the theory of 
$\mathbb Z_2$-graded C*-algebras, for the convenience of the reader we present here a direct proof af all of them in the ``spirit'' of Kre\u\i n 
C*-algebras. 

\begin{theorem}\label{tm2}
There is a rank-one unital symmetric Kre\u\i n C*-algebra.
\end{theorem}
\begin{proof}
For each $a$, $b \in \mathbb{C}$, let 
$T_{a,b} = \begin{bmatrix} a & b \\ b &a \end{bmatrix}.$ 
Let $\mathbb{K}=\big\{T_{a,b} : a, b \in \mathbb{C} \big\}$  with the usual matrix operations of addition and multiplication. 
We define the involution by $T_{a,b}^* = T_{\bar{a},-\bar{b}}$. Furthermore, we equip $\mathbb{K}$ 
with the operator norm and choose as fundamental symmetry $\gamma:\mathbb{K}\rightarrow\mathbb{K}$ defined by $\gamma(\begin{bmatrix} a & b \\ b &a \end{bmatrix})=\begin{bmatrix} a & -b \\ -b &a \end{bmatrix}$. Then we can write $\mathbb{K}=\mathbb{K}_+\oplus\mathbb{K}_-$ where 
\[  \mathbb{K}_+=\left\{ \begin{bmatrix} a & 0 \\ 0 &a \end{bmatrix}: a\in \mathbb{C}\right\} \quad \text{and} \quad \mathbb{K}_-
=\left\{ \begin{bmatrix} 0 & b \\ b &0 \end{bmatrix}: b\in \mathbb{C}\right\}. 
\]   
The symmetry of the algebra $\mathbb K$ can be checked defining $\epsilon(T_{a,b}):=T_{b,a}$. 
\end{proof}

\begin{theorem}\label{tm7}
The identity $\iota_{\mathbb{K}}$ and $\gamma$ are the only two unital $*$-automorphisms from $\mathbb{K}$ onto itself. Furthermore, $\gamma$ is the unique fundamental symmetry on $\mathbb{K}$. 
\end{theorem}
\begin{proof}
Let $\phi$ be the unital $*$-homomorphism from $\mathbb{K}$ into itself. 

Suppose that $\phi(e)=\phi(\begin{bmatrix}0& 1\\1 &0 \end{bmatrix})=\begin{bmatrix}a& b\\b &a \end{bmatrix}$ for some $a,b\in \mathbb{C}$.
Since $\phi(e)^2=1$ and $\phi(e^*)=-\phi(e)$, we have $(a,b)=(0,1)$ or $(a,b)=(0,-1)$. Then $\phi=\iota_{\mathbb{K}}$ or 
$\gamma$, and this proves the theorem. 
\end{proof}

We can characterize the rank-one unital Kre\u\i n \hbox{C*-algebras} as in Gel'fand-Mazur Theorem for \hbox{C*-algebras}. 
\begin{theorem}\label{tm3}
Every rank-one unital Kre\u\i n C*-algebra is isomorphic to $\mathbb{K}$.
\end{theorem}
\begin{proof}
By Gel'fand-Mazur Theorem, $A_+=\mathbb{C}1$. Since $\dim A_-=1$, there is a non-zero element $e$ such that $\langle e\rangle=A_-$. WLOG, we choose $e$ such that $\|e\|=1$.
Since $e^*\in A_-$ and $e\cdot e\in A_+$, we have $e^*=\alpha e$ and $e\cdot e=\beta$ for some $\alpha,\beta\in \mathbb{C}$.
From properties in Kre\u\i n C*-algebra, we can verify that $\alpha=\pm \exp(i\theta)$ and $\beta=\exp(i\theta)$ for some $\theta\in [0,2\pi)$. Now we have that \hbox{$e^*=\pm\exp(-i\theta)e$} and $e\cdot e=\exp(i\theta)$ for some $\theta\in [0,2\pi)$. We use the notation $A_{\pm \theta}$ corresponding with the preceding operations. 
Since $\|x\|=\|L_x\|$ where $L_x$ is the left multiplication operator, we can prove that \hbox{$\|x\|=\max\{|a+b|,|a-b|\}$}. It is easy to show that $\theta$ must be zero or $\pi$.

For case $\theta=\pi$, we let $x=y=i+e$, then
$$
\|xy\|=\max\{|-2+2i|,|-2-2i|\}=2\sqrt{2}\nleq 2=\max\{|i+1|^2,|i-1|^2\}=\|x\|\|y\|.
$$
Hence $A_{\pm\pi}$ is not a Kre\u\i n C*-algebra.

For $A_{-0}$, we have $e \cdot e=1$ and $e^*=-e$.
Since every element in $A_{0_-}$ is in the form $m1+ne$, where $m$ and $n$ are in $\mathbb{C}$, 
we define $f:A_{0_-}\rightarrow \mathbb{K}$ by 
$f(1)=\begin{bmatrix} 1 & 0 \\ 0 &1 \end{bmatrix}$ and $f(e)=\begin{bmatrix} 0 & 1 \\ 1 &0  \end{bmatrix}$. 
Then $f$ is an isomorphism between the Kre\u\i n C*-algebras $A_{0_-}$ and $\mathbb{K}$. 
Since $\mathbb{K}$ is Kre\u\i n C*-algebra, $A_{-0}$ is a Kre\u\i n C*-algebra too.

For $A_{+\theta}$, we have $e \cdot e=1$ and $e^*=e$. Suppose that $\gamma$ is the fundamental symmtry on it. 
Note that $\gamma(x^{*_-})=x^{*_+}$ for all $x\in A$ where $*_-$ and $*_+$ are involutive on $A_{-\theta}$ and $A_{+\theta}$, respectively.
Then
$$
\|x\|^2=\|\gamma(x^{*_+})x\|=\|x^{*_-}x\|,
$$
which contradicts to the fact that $A_-$ is the nontrivial Kre\u\i n C*-algebra. Since $\gamma$ is the only possible fundamental symmetry on $A$, we can conclude that $A_+$ is not a Kre\u\i n C*-algebra. 
\end{proof}

\begin{theorem}
The space $C(M,K)$ of all continuous functions from a compact Hausdorff space $M$ into a Kre\u\i n C*-algebra 
$K$ is a unital Kre\u\i n C*-algebra. 
\end{theorem}
\begin{proof}
Let $\gamma$ be a fundamental symmetry of the Kre\u\i n C*-algebra $K$ and $M$ a compact Hausdorff space.
We define all the operations and the norm as follows:
\begin{align*}
(f+g)(x)&=f(x)+g(x) 
& 
(fg)(x)&=f(x)g(x)
\\
(kf)(x)&=kf(x)
& 
f^{*_C}(x)&=f(x)^{*_K}\\
\|f\|_C&=\sup_{x\in M} \|f(x)\|_K,
& 
\end{align*}
for all $k\in \mathbb{C}$, $x\in M$ and for all $f,g\in C(M,K)$. 
It is easy to check that $C(M,K)$ is a Banach algebra and a $*$-algebra with the operations and norm above.

To see that $C(M,K)$ is a Kre\u\i n C*-algebra, consider the map $\phi_C$ from $C(M,K)$ into itself defined by 
$\phi_C(f)=\gamma\circ f$. Since, for all $x\in M$, 
\begin{align*}
\phi_C(fg)(x)&=\gamma\circ(fg)(x)=\gamma(f(x)g(x))=\gamma(f(x))\gamma(g(x))=\phi_C(f)\phi_C(g)(x) \\
\phi_C(f^{*_C})(x)&=(\gamma\circ f^{*_C})(x)=\gamma(f(x)^{*_K})=\gamma(f(x))^{*_K}\\
&=(\gamma\circ f)(x)^{*_K}=(\gamma\circ f)^{*_C}(x)=\phi_C(f)^{*_C}(x),   \\
\phi_C(1_C)(x)&=\gamma(1_C(x))=\gamma(1_K)=1_K=1_C(x),         	
\end{align*}
$\phi_C$ is  a unital $*$-homomorphism that is an involutive $*$-automorphism because 
$$
\phi_C\circ\phi_C(f)=\phi_C(\gamma\circ f)=\gamma\circ\gamma\circ f=f, \quad \forall f\in C(M,K), 
$$
and also a fundamental symmetry of a Kre\u\i n C*-algebra because 
\begin{align*}
\|\phi_C(f)^{*_C}f\|_C&=\sup_{x\in M} \|(\gamma\circ f)^{*_C}(x)f(x)\|_K
=\sup_{x\in M} \|(\gamma\circ f)(x)^{*_K}f(x)\|_K\\
&=\sup_{x\in M} \|\gamma(f(x))^{*_K}f(x)\|_K=\sup_{x\in M} \|f(x)\|^2_K=\|f\|^2_C,
\end{align*}
for all $f\in C(M,K)$. 
\end{proof}

When we take $K:=\mathbb K$ in the previous theorem, we obtain another example of commutative symmetric imprimitivity 
Kre\u\i n C*-algebra. 

\begin{corollary}
The Kre\u\i n C*-algebra $C(M,\mathbb K)$ is a commutative symmetric imprimitivity unital Kre\u\i n C*-algebra. 
\end{corollary}
\begin{proof}
Define the fundamental symmetry $\phi_C(f):=\gamma\circ f$ and $\epsilon_C(f):=\epsilon_{\mathbb K}\circ f$. 
\end{proof}

\begin{theorem}\label{tm4}
Let $(A,\alpha)$ be a (unital commutative) Kre\u\i n C*-algebra with a fundamental symmetry $\alpha$.
Let $I$ be a closed ideal in $A$ invariant under $\alpha$, i.e. $\alpha(I)\subseteq I$, then $A/I$ is also a (unital commutative) Kre\u\i n C*-algebra. 
\end{theorem}
\begin{proof}
Since $I$ is a closed ideal in the C*-algebra $(A,\dag_\alpha)$, the quotient $(A/I,\dag_\alpha)$ is a C*-algebra with involution 
$(x+I)^{\dag_\alpha}:=x^{\dag_\alpha}+I$. 
Since $I$ is invariant under the $*$-automorphism $\alpha$ we can define $[\alpha]:A/I\rightarrow A/I$ by $x+I\mapsto \alpha(x)+I$ for all $x\in A$ and $[\alpha]$ is a $\dag_\alpha$-automorphism of $(A/I,\dag_\alpha)$.  
By remark~\ref{re: c-k} $(A/I,*)$ is a Kre\u\i n C*-algebra with involution $(x+I)^*:=\alpha(x^{\dag_\alpha})+I=x^*+I$ and $[\alpha]$ becomes the fundamental symmetry on $(A/I,*)$. 
\end{proof}

\begin{theorem}\label{thm1}
Let $(A,\alpha)$ and $(B,\beta)$ be two Kre\u\i n C*-algebras with given fundamental symmetries $\alpha$ and $\beta$. 
A $*$-homomorphism $\phi:A\to B$ 
satisfying the property $\phi\circ\alpha=\beta\circ\phi$ is always continuous. Additionally, $\phi(A_+)\subseteq B_+$ and $\phi(A_-)\subseteq B_-$.
\end{theorem}
\begin{proof}
Note that $\phi$ is a $\dag$-homomorphism between the associated C*-algebras $(A,\dag_\alpha)$ and $(B,\dag_\beta)$.  
The ``invariance'' of $\phi$ under $\alpha,\beta$ implies the last property. 
\end{proof}
If $\phi:A\to B$ is a unital $*$-homomorphism such that $\phi\circ\alpha=\beta\circ\phi$ for a given pair of fundamental symmetries of the Kre\u\i n 
C*-algebras $A$ and $B$, in view of the last property in theorem~\ref{thm1}, we will denote by $\phi_+:A_+\to B_+$ and $\phi_-:A_-\to B_-$ the restrictions to the even and odd parts of the Kre\u\i n C*-algebras. Note that $\phi=\phi_+\oplus \phi_-$. 
In particular, the quotient isomorphism $\pi:A\to A/I$ from a Kre\u\i n C*-algebra A to its quotient Kre\u\i n C*-algebra by an ideal $I$ that is invariant under a fundamental symmetry $\alpha$ of $A$, can be written as a direct sum $\pi=\pi_+\oplus \pi_-$ of the isomorphisms $\pi_+:A_+\to (A/I)_+$ and 
$\pi_-:A_-\to (A/I)_-$. 

\begin{corollary}\label{tm5}
Let $(A,\alpha)$ be a unital Kre\u\i n C*-algebra with a fundamental symmetry $\alpha$. Let $w$ be a unital $*$-homomorphism from $A$ into 
$\mathbb{K}$ with the property that $w\circ\alpha=\gamma\circ w $. Then $A/\ker(w)$ is a Kre\u\i n C*-algebra. 
\end{corollary}
\begin{proof}
By the preceding theorm, $\ker(w)$ is a closed ideal in $A$. A direct computation shows that $\alpha(\ker(w))\subseteq \ker(w)$. 
Consequently, $A/\ker(w)$ is a Kre\u\i n C*-algebra by theorem~\ref{tm4}.
\end{proof}

\begin{corollary}\label{cor1}
Let $(A,\alpha)$ be a unital imprimitive Kre\u\i n C*-algebra with a fundamental symmetry $\alpha$. 
Let $w$ be a unital $*$-homomorphism from $A$ into $\mathbb{K}$ with the property that $w\circ\alpha=\gamma\circ w$. 
Then $A/\ker(w) \cong \mathbb{K}$. 
\end{corollary}
\begin{proof}
$A/\ker(w)$ is a Kre\u\i n C*-algebra by Corollary~\ref{tm5}. 

Define $\beta:A/\ker(w)\rightarrow \mathbb{K}$ by $x+\ker(w)\mapsto w(x)$ for all $x\in A$. It is easily checked that $\beta$ is unital injective 
$*$-homomorphism. It remains to show the surjective property. Since ${A_-}_{A_+}$ is full, WLOG 
$1=\sum_{j=1}^n x_j y_j$ for some $n\in \mathbb{N}$ and $x_j, y_j\in A_-$ for all $i=1,\ldots, n$. 
Suppose that $w(x)=0$ for all $x\in A_-$. Since $w$ is a unital homomorphism, we have 
$$
1=w(1)=w(\sum_{j=1}^n x_j y_j)=\sum_{j=1}^n w(x_j)w(y_j)=0,
$$
which leads to contradiction. Thus there is element $x$ in $A_-$ such that $w(x)\neq 0$. Since $w(A_-)\subseteq \mathbb{K}_-$ and $\dim{\mathbb{K}_-}=1$, we get $w(A_-)=\mathbb{K}_-$. Similarly, by Gel'fand-Mazur theorem 
$w(A_+)=\mathbb{K}_+$. Thus we already have
$$
\beta(A/\ker(w))=w(A)=w(A_+\oplus A_-)=w(A_+)\oplus w(A_-)=\mathbb{K}_+\oplus \mathbb{K}_-=\mathbb{K},
$$
and so $\beta$ is clearly surjective.
\end{proof}

\begin{definition}
A \textbf{character} on a unital (commutative symmetric full) Kre\u\i n C*-algebra $A$ is a unital $*$-homomorphism 
$w:A\rightarrow \mathbb{\mathbb{K}}$ such that there exists at least one fundamental symmetry $\alpha$ of the Kre\u\i n C*-algebra $A$ satisfying the property $w\circ\alpha=\gamma\circ w$. In this case we will say that $\alpha$ and $w$ are \emph{compatible}. 
\end{definition}
We denote by $\Omega(A)$ the set of characters on $A$, and by $\Omega(A,\alpha)$ the set of characters compatible with a given fundamental symmetry $\alpha$ that is
\begin{align*}
\Omega(A,\alpha)=\{w\mid w:A\rightarrow \mathbb{K} \text{ unital $*$-homomorphism compatible with $\alpha$}\}. 
\end{align*}
For each $a\in A$, define $\hat{a}:\Omega(A)\rightarrow \mathbb{K}$ by $\hat{a}(w)=w(a)$ for all $w\in \Omega(A)$.
Equip $\Omega(A)$ with the smallest topology which makes each $\hat{a}$ continuous. 

Note that $\Omega(A,\alpha)$ with the induced subspace topology is clearly a closed set in the topological space $\Omega(A)$ since 
\hbox{$\Omega(A,\alpha)=\{w\in \Omega(A) \mid \widehat{\alpha(x)}(w)=\gamma\circ\hat{x}(w), \forall x\in A\}$} and the functions 
$ \widehat{\alpha(x)}, \gamma\circ\hat{x}$ are continuous on $\Omega(A)$. 

We next define an equivalence relation between the characters as follows:
\begin{equation*}
w_1\sim w_2 \Longleftrightarrow w_2=w_1\circ \phi \text{ for some } \phi \text{ a unital $*$-automorphism on } \mathbb{K}.
\end{equation*}
By Theorem \ref{tm7}, $[w]=\{w, \gamma\circ w\}$. Note that $w\in\Omega(A,\alpha)$ implies $[w]\subset\Omega(A,\alpha)$. 
\begin{equation*}
\text{We define} \quad 
\Omega_b(A)=\{[w]\mid w\in \Omega(A)\}, \quad \Omega_b(A,\alpha)=\{[w]\mid w\in \Omega(A,\alpha)\}. 
\end{equation*}
Equip $\Omega_b(A)$ with the quotient topology induced by the quotient map 
$$\mu:\Omega(A)\rightarrow \Omega_b(A) \ \text{given by }\  \mu: w\mapsto [w].$$
Note that on $\Omega_b(A,\alpha)$ the quotient topology induced by $\Omega(A,\alpha)$ coincides with the subspace topology induced by 
$\Omega_b(A)$. 

\begin{lemma}\label{lem: comp}
$\Omega(A)$ and $\Omega(A,\alpha)$ are compact Hausdorff spaces. 
\end{lemma}
\begin{proof}
Since a character $w:A\to\mathbb K$ is a unital $*$-homomorphism of the associated \hbox{C*-algebras} $(A,\dag_\alpha)$ and 
$(\mathbb K,\dag_\gamma)$, with the same techniques used in Banach-Alaoglu theorem, it is a standard matter to check that 
$\Omega(A)$ is a closed subset of compact set $\prod_{x\in A} \overline{B(0_\mathbb{K}, \|x\|)}$ and hence 
$\Omega(A)$ is also a compact set. To show that $\Omega(A)$ is a Hausdorff space, we first let $w_1, w_2$ be characters such that $w_1\neq w_2$. Then there is $a\in A$ such that $w_1(a)\neq w_2(a)$ that is $\hat{a}(w_1)\neq\hat{a}(w_2)$. Since $\hat{a}$ is continuous on $\Omega(A)$ and 
$\mathbb{K}$ is Hausdorff, we obtain that $\Omega(A)$ is also Hausdorff. 
Since $\Omega(A,\alpha)$ is a closed subspace of $\Omega(A)$, the result follows.
\end{proof}

\begin{lemma}\label{lem2} Let $w, w_1$ and $w_2$ be characters. The following properties hold.
\begin{enumerate}
	\item[a)] If $w_1\sim w_2$, then $\ker(w_1)=\ker(w_2)$.
	\item[b)] If $w\in \Omega(A,\alpha)$, $w(x)=0\Longleftrightarrow w(x^{\dag_\alpha}x)=0$, for all $x\in A$.
	\item[c)] If $w_1,w_2\in \Omega(A,\alpha)$ and ${w_1}_+={w_2}_+$, then $\ker(w_1)=\ker(w_2)$.
\end{enumerate}
\end{lemma}
\begin{proof}Property a) is clear. 
For b), assume that $w(x^{\dag_\alpha}x)=0$. 
Then
$$
\|w(x)\|^2=\|\gamma(w(x))^*w(x)\|=\|w(\alpha(x))^*w(x)\|=\|w(\alpha(x^*)x)\|=\|w(x^{\dag_\alpha}x)\|=0,
$$
and so $w(x)=0$.

For c), by the assumption, we have $\ker({w_1}_+)=\ker({w_2}_+)$. 
Let $x\in A_-$ be such that $w_1(x)=0$. By b), $w_1(\alpha(x^*)x)=0$ and also $w_2(\alpha(x^*)x)=0$ because 
$\ker({w_1}_+)=\ker({w_2}_+)$. Again by b), $w_2(x)=0$ and hence $\ker({w_1}_-)\subseteq\ker({w_2}_-)$. 
By the same argument, it is elementary to verify the inverse inclusion. 
\end{proof}

\begin{lemma}\label{lem1}
Let $(A,\alpha)$ be a Kre\u\i n C*-algebra equipped with a fundamental symmetry $\alpha$ and let $\omega\in \Omega(A_+)$ be a character defined on the even part of the Kre\u\i n C*-algebra $A=A_+\oplus A_-$.  
Define  
  $$I_+:=\ker(w), \quad I_-:=A_-\ker(w):=\Span\{xa \mid x\in A_-, \ a\in \ker{w}\}, \quad I:=I_+\oplus I_-.$$ 
  Then $I$ is an ideal in $A$ invariant under $\alpha$ and the following properties hold
	\begin{enumerate}
		\item[a)] $A_+/I_+$ is a C*-algebra with dimension 1.
		\item[b)] $A_-/I_-$ is a Hilbert C*-bimodule over $A_+/I_+$.
		\item[c)] $A/I=(A/I)_+\oplus (A/I_-)\cong A_+/I_+\oplus A_-/I_-$.
	 	\item[d)] $A/I$ is a rank-one Kre\u\i n C*-algebra.
	\end{enumerate}
\end{lemma}
\begin{proof}
The bimodule $A_-/I_-$ over $A_+/I_+$ is a Hilbert C*-bimodule with the inner products defined as in the proof of theorem~\ref{tm1} using the 
$(x+I_-)^{\dag_\alpha}:=x^{\dag_\alpha}+I_-$ involution. 
The only other thing that is not completely straightforward is that $A/I$ is rank 1. By b), since by Gel'fand-Mazur 
$A_+/I_+\cong \mathbb C$, $A_-/I_-$ is a Hilbert space over $A_+/I_+$. To show the rank-one property, suppose by contradiction that 
$x,y\in A_-/I_-$ is a pair of orthonormal vectors, then 
$$
y=\langle x|x\rangle y+\langle y|x\rangle y=\langle x+y|x\rangle y=(x+y)\langle x|y \rangle=0,
$$
which is impossible. 
\end{proof}

\begin{theorem}
If $(A,\alpha)$ is a unital commutative imprimitive Kre\u\i n C*-algebra with fundamental symmetry $\alpha$, then 
$\Omega_b(A,\alpha)$ is a compact Hausdorff space.
\end{theorem}
\begin{proof}
Since $\Omega(A,\alpha)$ is compact by lemma~\ref{lem: comp} and $\mu$ is continuous, $\Omega_b(A,\alpha)$ is also compact.

We now consider the map $\phi$ from $\Omega_b(A,\alpha)$ to $\Omega(A_+)$ defined by $[w]\mapsto w_+$. If we can show that this map is a homeomorphism, we can conclude that $\Omega_b(A,\alpha)$ is a Hausdorff space since $\Omega(A_+)$ is a Hausdorff space by the spectral theorem for unital commutative C*-algebras. It is sufficient to show that $\phi$ is a continuous bijective map because $\Omega_b(A,\alpha)$ is a compact space and $\Omega(A)$ is a Hausdorff space. 

Firstly, $\phi$ is well-defined. More precisely, let $[w_1]=[w_2]$, if $w_1\neq w_2$, then $w_2=\gamma\circ w_1$. Let $x\in A_+$,  then 
$
w_1(x)=w_1(\alpha(x))=\gamma(w_1(x))=w_2(x)
$, 
that is, ${w_1}_+={w_2}_+$.

We next show that $R:=\phi\circ\mu$ is a continuous map. Note that, by definition, we have $R(w)=w_+$.  For easier consideration, we provide a diagram of all the functions involved,
\[
\xymatrix{
& \Omega(A,\alpha) \ar[ld]_\mu \ar[rd]^{\hat{a}} \ar[d]_R& 
\\ 
 \Omega_b(A,\alpha)\ar[r]_\phi & \Omega(A_+)  \ar[r]_{\hat{a}_+} & \mathbb C
}
\]
since for all $a\in A_+$, $\hat{a}=\hat{a}_+\circ R$ is continuous on $\Omega(A,\alpha)$, $R$ is also continuous. 
Since $R$ is a continuous map, by the properties of quotient topology we also have that $\phi$ is a continuous map. 
Next, we will show that $\phi$ is a bijection.

To show that $\phi$ is injective, we suppose that $w_1, w_2$ are characters on $(A,\alpha)$ such that ${w_1}_+={w_2}_+$. 
We next examine the diagram here 
\begin{equation*}
\xymatrix{
A \ar[r]^{w_1} \ar[d]_{\pi_1} & \KK & A \ar[l]_{w_2} \ar[d]^{\pi_2}
\\
A/\ker(w_1)\cong \mathbb K \ar[ur]_{\beta_1} & & A/\ker(w_2)\cong \mathbb K. \ar[ul]^{\beta_2}
}
\end{equation*}

By lemma~\ref{lem2} (c) and corollary~\ref{cor1} we have $A/\ker(w_1)\cong \mathbb K\cong A/\ker(w_2)$, 
$\beta_2\circ\beta_1^{-1}$ is a unital $*$-automorphism on $\mathbb{K}$. Since by theorem~\ref{tm7} a unital $*$-automorphism on $\mathbb{K}$ is either the identity map or $\gamma$, we have to consider two cases.
\begin{enumerate}
	\item $\beta_2\circ\beta_1^{-1}$  is an identity on $\mathbb{K}$\\
	Then $\beta_1=\beta_2$. Since $\ker(w_1)=\ker(w_2)$, we have
	$$
	w_1(a)=\beta_1(a+\ker{w_1})=\beta_2(a+\ker{w_2})=w_2(a),
	$$
	for all $a\in A$. Thus $w_1=w_2$.
	\item $\beta_2\circ\beta_1^{-1}=\gamma$\\
	Then $\beta_2=\gamma\circ\beta_1$. Again by the fact that $\ker(w_1)=\ker(w_2)$, we have
	$$
	w_2(a)=\beta_2(a+\ker{w_2})=\gamma\circ\beta_1(a+\ker{w_1})=\gamma\circ w_1(a),
	$$
	for all $a\in A$.
	Hence $w_2=\gamma\circ w_1$.
\end{enumerate}
From both of two cases, we can conclude that $[w_1]=[w_2]$ which implies the injection of $\phi$, as we wanted.
To prove the surjectivity of $\phi$, let $w^o: A_+\rightarrow \mathbb{K}_+$. 
Consider $I_+:=\ker(w^o)$ and define the ideal $I$ in $A$ as in lemma \ref{lem1}. 
 
Since $A/I$ is a rank-one Kre\u\i n C*-algebra, $A/I$ is isomorphic to $\mathbb K$ by an isomorphism $f$ such that $f\circ [\alpha]=\gamma\circ f$, where $[\alpha]$ is the fundamental symmetry of $A/I$ such that $\pi\circ\alpha=[\alpha]\circ \pi$. 
Note that $f=f_+\oplus f_-$ where $f_+:(A/I)_+\to\mathbb K_+$ is the isomorphism $f_+(a+I)=w^o(a)$, for all $a\in A_+$. 
\begin{equation*}
\xymatrix{
A \ar[r]^-\pi & {A/I} \ar[r]^-f & 
\mathbb K. 
}
\end{equation*}

Define $w:=f\circ \pi$. Then $w\in \Omega(A,\alpha)$ because $w\circ\alpha=\gamma\circ w$. We claim that $w_+=w^o$. To see this, let $a\in A_+$. Then $w(a)=f\circ\pi(a)=f(a+I)=w^o(a)$, and the theorem is proved.
\end{proof}

\begin{definition}
If $(A,\alpha,\epsilon)$ is a commutative full symmetric unital Kre\u\i n C*-algebra with a given fundamental symmetry $\alpha$ and a given odd symmetry $\epsilon$, we define the \emph{even spectrum} as the set of \emph{even characters} of $A$: 
$\Omega(A,\alpha,\epsilon):=\{w\in \Omega(A,\alpha) \mid \epsilon_{\mathbb K}\circ w\circ \epsilon=w\}$. 
\end{definition}

\begin{theorem}\label{prep}
If $(A,\alpha, \epsilon)$ is a commutative unital full symmetric Kre\u\i n C*-algebra with fundamental symmetry $\alpha$ and odd symmetry $\epsilon$, the character $w\in \Omega(A,\alpha)$ is even if and only if $\gamma\circ w\in \Omega(A,\alpha)$ is not even. Hence in every equivalence class 
$[w]=\{w,\gamma\circ \omega\}$ there is one and only one even character and there is a bijection between $\Omega(A,\alpha,\epsilon)$ and 
$\Omega_b(A,\alpha)$. 
\end{theorem}
\begin{proof}
Since $\epsilon_{\mathbb K}\circ \gamma=-\gamma\circ \epsilon_{\mathbb K}$, we see that $w$ is even if and only if $\gamma\circ w$ is odd. 
Let $w^o$ be a character on $A_+$ and define $w(x):=w^o(x_+)+\epsilon_{\mathbb K}\circ w^o\circ \epsilon(x_-)$. 
By the properties of the odd symmetry, we see that $w$ is an even character compatible with $\alpha$. 
\end{proof}

\begin{definition}
Let $A$ be a commutative unital full symmetric Kre\u\i n C*-algebra and let $\alpha$ be a fundamental symmetry of $A$ and 
$\epsilon$ an odd symmetry of $A$. 
The \emph{Gel'fand transform} of $x\in A$ is the map $\hat{x}:\Omega_b(A,\alpha)\to\mathbb K$ defined by:
\begin{equation*}
\hat{x}([w]):=w(x_+)+\epsilon_{\mathbb K}\circ w\circ \epsilon (x_-), \quad \forall x\in A. 
\end{equation*}
\end{definition} 
By the previous theorem, it is clear that the Gel'fand transform $\hat{x}$ of $x$ is just the function that to every even character 
$w\in \Omega(A,\alpha,\epsilon)$ associates $w(x)$. 

Although the following theorem is our goal, the proof is easy and straightforward.
\begin{theorem}[Spectral theorem]\label{thm2}
If $(A,\alpha,\epsilon)$ is a unital commutative symmetric imprimitive Kre\u\i n C*-algebra with fundamental symmetry $\alpha$ and odd symmetry 
$\epsilon$, then the Gelfand transform
$$\varphi:A\rightarrow C(\Omega_b(A,\alpha),\mathbb{K}),~a\mapsto \hat{a}$$
is an isometric $*$-isomorphism.
\end{theorem}
\begin{proof}
In view of theorem~\ref{prep} let us denote by $[w]$ with $w$ even a point of $\Omega_b(A,\alpha)$.  
To prove that $\varphi$ is a $*$-homomorphism of algebras, let $a,b\in A$ and $k\in\mathbb{C}$,
\begin{align*}
&(\varphi(ab))([w])=(\widehat{ab})([w])=w(ab)=w(a)w(b)=\hat{a}([w])\hat{b}([w])=(\varphi(a)\varphi(b))([w]), \\
&(\varphi(ka+b))([w])= (\widehat{ka+b})([w])=w(ka+b)=kw(a)+w(b)=k\hat{a}([w])+\hat{b}([w]) \\
                   & \qquad  =(k\varphi(a)+\varphi(b))([w]),               \\
&\varphi(a^*)[w]=\widehat{a^*}([w])=[w](a^*)=w(a)^*=\varphi(a)^*[w].                    
\end{align*}
for all $[w]\in \Omega_b(A,\alpha)$. Clearly $\varphi(1_A)=1_C$ so that $\varphi$ is unital.

It is easy to verify that $\varphi\circ\alpha=\phi_C\circ\varphi$. In fact, for all $a\in A, [w]\in \Omega_b(A,\alpha)$,
\begin{align*}
(\varphi\circ\alpha)(a)[w]&=\varphi(\alpha(a))[w]=w(\alpha(a))=\gamma(w(a))=\gamma(\varphi(a)[w])\\
                          &=(\gamma\circ\varphi(a))[w]=\phi_C(\varphi(a))[w]=(\phi_C\circ\varphi)(a)[w].
\end{align*}

We also have that $\varphi\circ\epsilon=\epsilon_C\circ\varphi$, in fact, for all $[w]\in \Omega_b(A,\alpha)$, and all $x\in A$,  
\begin{equation}\label{eqepsvar}
\epsilon_C\circ\varphi(x)[w]=\epsilon_{\mathbb K}(w(x))=\epsilon_{\mathbb K}\circ w\circ \epsilon^2(x)=w\circ\epsilon(x)=\varphi\circ \epsilon(x)[w].  
\end{equation}

Let $A$ be a unital commutative Kre\u\i n C*-algebra with the fundamental symmetry $\alpha$. Then $(A,\dag_A)$ and 
$(C(\Omega_b(A,\alpha),\mathbb{K}),\dag_C)$ become C*-algebras with the involutions $\dag_A$ and $\dag_C$ defined as in theorem~\ref{tm6} respectively, that is,
$$
a^{\dag_A}=\alpha(a^{*_A}) \text{ for all } a\in A \text{ and } f^{\dag_C}=\phi_C(f^{*_C}) \text{ for all } f\in C(\Omega_b(A,\alpha),\mathbb{K}).
$$  
To show $\varphi$ is a $\dag$-homomorphism, let $a\in A$ and $[w]\in \Omega_b(A,\alpha)$,
$$
\varphi(a^{\dag_A})[w]=\varphi(\alpha(a^{*_A}))[w]=\widehat{\alpha(a^{*_A})}[w]=w(\alpha(a^{*_A})).
$$
Since $w\circ\alpha=\gamma\circ w$, we have,
$$
\varphi(a^{\dag_A})[w]=\gamma(w(a^{*_A}))=\gamma(w(a)^{*_\mathbb{K}})
=\gamma(\hat{a}([w])^{*_\mathbb{K}})=\gamma(\hat{a}^{*_C}([w]))
$$
Since $\phi_C(f)=\gamma\circ f$ for all $f\in C(\Omega_b(A),\mathbb{K})$,
$$
\varphi(a^{\dag_A})[w]=\phi_C(\hat{a}^{*_C})[w]=\hat{a}^{\dag_C}[w]=\varphi(a)^{\dag_C}[w].         
$$

By the spectral theorem for unital commutative C*-algebras, the restriction of the Gel'fand transform to the even part 
$\varphi_+:A_+\to C(\Omega(A,\alpha),\mathbb K)_+$ is an isometric 
$\dag$-isomorphism, for all $a\in A$ that coincides with the usual Gel'fand isomorphism for the commutative unital C*-algebra $A_+$. 

Since $\epsilon$ and $\epsilon_C$ are linear surjective (because $\epsilon^2=i_A$) isometries, from the equation \eqref{eqepsvar}, we see that 
$\varphi_-=\epsilon_-\circ\varphi_+\circ{\epsilon_C}_+$ is isometric surjective too and hence $\varphi=\varphi_+\oplus\varphi_-$ is a surjective isometry that is also a $*$-homomorphism.
\end{proof}

The previous theorem provide a complete characterization of those unital commutative Kre\u\i n \hbox{C*-algebras} that are full and symmetric. In this case, a posterori, with a bit more work, we might actually prove that the fundamental symmetry and the odd symmetry are indeed unique. 
Further analysis is required in order to provide a spectral theory of more general Kre\u\i n C*-algebras. Omitting the exchange symmetry requirement will lead us to algebras of sections of bundles of one-dimensional Kre\u\i n C*-algebras and omitting the commutativity (or just the imprimitivity condition on the odd part) will lead us to a theory of Kre\u\i n spaceoids along very similar lines to those used to describe the spectrum of commutative full C*-categories in~\cite{Be}, we hope that the extra effort paid here to describe a ``direct proof'' of the spectral theory in the special case of commutative full symmetric Kre\u\i n C*-algebras will facilitate the analysis of those more general topics that we plan to address in the near future.

\vskip.5cm \noindent{\bf Acknowledgements :} We thank Dr. Roberto Conti for making several suggestions and providing references to arguments in the theory of C*-dynamical systems ($\mathbb Z_2$-graded C*-algebras) that might be used to shortcut/simplify most of the proofs and results provided in this work. We plan to return on this topic and to improve the manuscript in this direction.

\end{document}